\documentclass{article}
\usepackage[british]{babel}
\usepackage{amsmath,amssymb,latexsym,array}
\usepackage[margin=0.5in]{geometry}
\usepackage{longtable}

\newcommand{\tmmathbf}[1]{\ensuremath{\boldsymbol{#1}}}
\newcommand{\tmop}[1]{\ensuremath{\operatorname{#1}}}
\newcommand{\tmtextbf}[1]{\text{{\bfseries{#1}}}}
\newcommand{\tmtextit}[1]{\text{{\itshape{#1}}}}
\newcommand{\tmtextup}[1]{\text{{\upshape{#1}}}}
\newenvironment{itemizedot}{\begin{itemize} }{\end{itemize}}
\newenvironment{proof}{\noindent\textbf{Proof\ }}{\hspace*{\fill}$\Box$\medskip}
\newtheorem{conjecture}{Conjecture}
\newtheorem{definition}{Definition}
\newtheorem{lemma}{Lemma}
\newtheorem{proposition}{Proposition}

\begin{document}

\title{A conjecture concerning $\ast$-algebras that unifies some matrix
decompositions}

\author{
  Ran Gutin\\
  Imperial College London
}

\maketitle

\begin{abstract}
  In this note, we propose a simple-looking but broad conjecture about
  star-algebras over the field of real numbers. The conjecture enables many
  matrix decompositions to be represented by star-algebras and star-ideals.
  This paper is written for people with a background in representation theory
  and module theory. The motivation for investigating this is the possibility
  of expressing polymorphic algorithms in numerical and theoretical linear
  algebra. This is similar to but different from algebraic (semiring based)
  approaches to dynamic programming. We prove certain cases of the conjecture.
\end{abstract}

\section{Introduction}

In this note, we propose a conjecture about finite-dimensional $\ast$-algebras
(also called involution algebras) {\cite{schmudgen2020invitation}} which we
think has applications in numerical and theoretical linear algebra. Our
conjecture is that a simple-looking but broad generalisation of the spectral
theorem is true for all finite-dimensional $\ast$-algebras over the field
$\mathbb{R}$. We provide a proof for some cases.

We state the conjecture:

\begin{conjecture}
  Let $H$ be a self-adjoint element (that is, satisfying $H^{\ast} = H$) of a
  finite-dimensional $\ast$-algebra $\mathcal{A}$ over $\mathbb{R}$. Let $k
  \in \mathbb{N}$ be the maximum number for which the following decomposition
  exists:
  \[ H = \sum_{i = 1}^k P_i H P_i, \]
  where for all $i, j \in [k]$:
  \begin{itemizedot}
    \item $P_i \in \mathcal{A},$
    
    \item $P_i^{\ast} = P_i,$
    
    \item $P_i P_j = \delta_{i j} P_i,$
    
    \item $P_i \neq 0$.
  \end{itemizedot}
  Consider another such decomposition for the same $H$:
  \[ H = \sum_{i = 1}^k Q_i H Q_i . \]
  Then there exists a permutation $\sigma \in S_k$ and a $U \in
  \mathcal{A}$ such that:
  \begin{itemizedot}
    \item $U^{\ast} = U^{- 1}$,
    
    \item $U H U^{\ast} = H$,
    
    \item for all $i \in [k]$: $U Q_i U^{\ast} = P_{\sigma (i)} .$
  \end{itemizedot}
\end{conjecture}

By assuming the conjecture, we can pick $\ast$-algebras which in some sense
\tmtextit{represent} certain matrix decompositions (with prior investigation
here {\cite{gutin2023unitary}}). By a matrix decomposition, we mean a way of
writing matrices as a product of other matrices -- a connection which is
partly justified by lemma \ref{additive-implies-multiplicative}. This
unification of some matrix decompositions emphasises the uniqueness aspect, as
opposed to the existence aspect (which is made trivial), of those
decompositions. One of the motivations is that the use of $\ast$-algebras
allows computer code written to compute \tmtextit{one} decomposition to be
directly used to compute different ones (related to but different from the somewhat well-known ideas in 
{\cite{shaikhha2019finally,huang2008advanced}}). Note that we will not
develop these computing applications here because this is a theoretical paper.

We demonstrate the correspondence this creates using some examples below.

\section{Consequences for different $\ast$-algebras}

We're going to show how this conjecture easily re-derives some existing matrix
decompositions. These re-derivations are justified by the observation (later
proven in lemma \ref{additive-implies-multiplicative}) that the conjecture is
sometimes equivalent to the statement that over certain $\ast$-algebras, the
conjecture is equivalent to the cancellation property of unitary similarity.

\subsection{Matrices over the complex numbers, with conjugate-transpose as
their involution}

Let $H$ be an $n \times n$ Hermitian matrix over $\mathbb{C}$. We recall the
spectral theorem from linear algebra:
\[ H = V D V^{\ast}, \]
and recall that we're interested in decompositions of the form
\[ H = \sum_{i = 1}^k P_i H P_i, \]
for maximum $k$. In fact, we may obtain this by taking each column of $V$,
which we will call $v_i$, and letting $P_i = v_i v_i^{\ast}$. We get $k = n$
and that $P_i H P_i = \lambda_i v_i v_i^{\ast}$. So we have:
\[ H = \sum_{i = 1}^k P_i H P_i = \sum_{i = 1}^n \lambda_i v_i v_i^{\ast} . \]
This illustrates (and sketches the proof for) the conjecture for $\mathbb{C}$,
and shows that the conjecture generalises the spectral theorem.

\subsection{Matrices over the double numbers, with conjugate-transpose as
their involution}

Let ${}^2\mathbb R$ denote the $*$-algebra of \emph{double numbers}: The underlying algebra is $\mathbb R \oplus \mathbb R$, and the involution is $(a,b)\mapsto(b,a)$.

Given any finite-dimensional (\tmtextit{without} star) $\mathbb{R}$-algebra
$R$, we can make a $\ast$-algebra $R \oplus R^{\tmop{op}}$, with involution
being $(x, y) \mapsto (y, x)$ \cite{hahn2013classical}. Note that for some such algebras $R$ (but not all), there exists some (any) involution, and we therefore have $R^{\tmop{op}} \cong R$. In those cases, we have $R \oplus R^{\tmop{op}} \cong R \oplus R \cong R \otimes {}^2\mathbb R$, where the algebra $R$ in the last expression is equipped with any involution.

Observe that the conjecture for $M_n (R) \oplus M_n (R)^{\tmop{op}}$ is
equivalent to the invariant subspace decomposition for matrices in $M_n (R)$. The existence of such a decomposition is a corollary of the Krull-Schmidt theorem {\cite{basicalgebra}} -- albeit for $R \in \{ \mathbb{R}, \mathbb{C}, \mathbb{H} \}$ it is better described as the Jordan Normal Form or primary decomposition.

In particular, this means that the conjecture for $M_n ({}^2\mathbb R)$ -- that is, $n \times n$ matrices over the double numbers -- is
equivalent to the Jordan Normal Form for $\mathbb{R}$-matrices.

\subsection{General picture}

In general, we obtain a correspondence between on the one hand:
\begin{itemizedot}
  \item Matrix decompositions,
\end{itemizedot}
and
\begin{itemizedot}
  \item finite-dimensional $\ast$-algebras over $\mathbb{R}$ along with an
  ideal.
\end{itemizedot}

The corresponding decomposition is obtained by considering Hermitian matrices over the given $*$-algebra whose elements all belong to the specified ideal, and then (nearly all the time) considering block-diagonal canonical forms of those matrices under unitary similarity.

We make this explicit using a table, where we exhaust all indecomposable 2 and
3-dimensional $\ast$-algebras (which can be done with the aid of the
Wedderburn-Malcev theorem) and consider the decomposition which the conjecture implies for all of its matrix algebras:

\begin{longtable}{m{4cm}|m{3cm}|m{1.25cm}|m{7cm}|m{1.25cm}}
  \hline
  Algebra & Involution & $\ast$-ideal & Corresponding decomposition & Tame?\\
  \hline\hline 
  $\mathbb{R}$ & $\tmop{id}$ & $\langle 1 \rangle$ & Spectral theorem & Y\\
  \hline
  $\mathbb{C}$ & $a + b i \mapsto a - b i$ & $\langle 1 \rangle$ & Spectral theorem for $\mathbb{C}$-Hermitian matrices & Y\\
  \hline
  $\mathbb{C}$ & $\tmop{id}$ & $\langle 1 \rangle$ & Complex-symmetric spectral theorem & Y\footnote{follows from observation that all square $\mathbb{C}$-matrices are similar to $\mathbb{C}$-symmetric matrices, combined with the polar decomposition trick given for instance in \cite{kaplansky2003linear}}\\
  \hline
  $\mathbb{R} [X] / (X^2)$ & $\tmop{id}$ & $\langle 1 \rangle$ & Spectral decomposition of $H + \varepsilon H' + o (\varepsilon)$ where $H, H' \in M_n(\mathbb R)$ and both are symmetric & Y \cite{perturbationanalytic, dualnumber}\\
  \hline
  $\mathbb{R} [X] / (X^2)$ & $a + b X \mapsto a - b X$ & $\langle 1 \rangle$ & Spectral decomposition of $H + \varepsilon H' + o (\varepsilon)$ where $H, H' \in M_n(\mathbb R)$; $H$ is symmetric; $H'$ is skew-symmetric & Y \cite{perturbationanalytic, dualnumber}\\
  \hline
  $\mathbb{R} [X] / (X^2)$ & $a + b X \mapsto a - b X$ & $\langle X \rangle$ & Spectral theorem for skew-symmetric $\mathbb{R}$-matrices & Y\\
  \hline
  $\mathbb{R} \oplus \mathbb{R}$ & $(a, b) \mapsto (b, a)$ & $\langle 1 \rangle$ & Jordan Normal Form & Y\\
  \hline
  $(\mathbb{R} \oplus \mathbb{R}) + \mathbb{R} \delta$ \footnote{\label{threed-arithmetic}with relations $\delta^2 = 0$, $\delta(x,y) = (y,x)\delta = x\delta$}  & Unique one where $\delta \mapsto \delta$ & $\langle 1 \rangle$ & Canonical basis for a pair $(Q,L)$ consisting of a quadratic form $Q$ and an operator $L$ self-adjoint with respect to $Q$ & Y {\cite{gohberg2006indefinite}}\\
  \hline
  $(\mathbb{R} \oplus \mathbb{R}) + \mathbb{R} \delta$ \textsuperscript{\ref{threed-arithmetic}}  & Unique one where $\delta \mapsto - \delta$ & $\langle 1 \rangle$ & Canonical basis for a pair $(\omega,L)$ consisting a of symplectic form $\omega$ and an operator $L$ self-adjoint with respect to $\omega$ & Y {\cite{gohberg2006indefinite}}\\
  \hline
  $(\mathbb{R} \oplus \mathbb{R}) + \mathbb{R} \delta$ \textsuperscript{\ref{threed-arithmetic}} & Unique one where $\delta \mapsto \delta$ & $\langle \delta \rangle$ & Sylvester's Law of Inertia & Y\\
  \hline
  $(\mathbb{R} \oplus \mathbb{R}) + \mathbb{R} \delta$ \textsuperscript{\ref{threed-arithmetic}}  & Unique one where $\delta \mapsto - \delta$ & $\langle \delta \rangle$ & Analogue of Sylvester's Law of Inertia for skew-symmetric matrices\footnote{The decomposition is for when $M = -M^T \in M_n(\mathbb R)$; then there is an invertible $P \in M_n(\mathbb R)$ such that $P M P^T = \begin{bmatrix}0 & -1 \\ 1 & 0\end{bmatrix}^{\oplus \tmop{rank}(M)/2} \oplus [0]^{\oplus\tmop{null}(M)}$} & Y\\
  \hline
  $\mathbb{R} [X] / (X^3)$ & $\tmop{id}$ & $\langle 1 \rangle$ & 2nd-order perturbation theory of spectral theorem & Y \cite{perturbationanalytic}\\
  \hline
  $\mathbb{R} [X] / (X^3)$ & $X \mapsto - X$ & $\langle 1 \rangle$  & 2nd-order perturbation theory of invariant subspace decomposition where skew-symmetric part is infinitesimal & Y \cite{perturbationanalytic}\\
  \hline
  $\mathbb{R} [X] / (X^3)$ & $X \mapsto - X$ & $\langle X \rangle$  &  Spectral decomposition of $H + \varepsilon H' + o (\varepsilon)$ where $H, H' \in M_n(\mathbb R)$;  $H$ is skew-symmetric; $H'$ is symmetric & Y\\
  \hline
  $\mathbb{R} [X, Y] / (X^2, Y^2, X Y)$ & $X \mapsto X, Y \mapsto - Y$ & $\langle 1 \rangle$ & Spectral decomposition of $H + \varepsilon H' + o (\varepsilon)$ where $H, H' \in M_n(\mathbb R)$; $H$ is symmetric; $H'$ is an \emph{arbitrary} matrix & N\footnote{\label{wildeven}Wild even when restricted to an ideal}\\
  \hline
  $\mathbb{R} [X, Y] / (X^2, Y^2, X Y)$ & $X \mapsto X, Y \mapsto Y$ & $\langle 1 \rangle$ & 1st-order perturbation theory of spectral theorem, with 2 independent perturbations in symmetric directions & N\textsuperscript{\ref{wildeven}}\\
  \hline
  $\mathbb{R} [X, Y] / (X^2, Y^2, X Y)$ & $X \mapsto - X, Y \mapsto - Y$ & $\langle 1 \rangle$ & 1st-order perturbation theory of spectral theorem, with 2 independent perturbations in skew-symmetric directions & N\textsuperscript{\ref{wildeven}}\\
  \hline
  $\mathbb{R} [X, Y] / (X^2, Y^2, X Y)$ & $X \mapsto X, Y \mapsto - Y$ & $\langle X, Y \rangle$ & Block-diagonal form for $\mathbb{R}$-matrices under orthogonal similarity & N {\cite[sec. 4]{sergeichuk1998unitary}}\\
  \hline
  $\mathbb{R} [X, Y] / (X^2, Y^2, X Y)$ & $X \mapsto X, Y \mapsto Y$ & $\langle X, Y \rangle$ & Block-diagonal form for pairs of symmetric $\mathbb{R}$-matrices under orthogonal similarity. & N {\cite[sec. 4]{sergeichuk1998unitary}}\\
  \hline
  $\mathbb{R} [X, Y] / (X^2, Y^2, X Y)$ & $X \mapsto - X, Y \mapsto - Y$ & $\langle X, Y \rangle$ & Block-diagonal form for pairs of skew-symmetric $\mathbb{R}$-matrices under orthogonal similarity & N {\cite[sec. 4]{sergeichuk1998unitary}}\\
  \hline
\end{longtable}
The above lists all 2- and 3-dimensional cases. Note that we treat much of the
perturbation theory of matrix decompositions as -- in some sense -- being
matrix decompositions in their own right.

We also consider some notable 4-dimensional cases. There are too many cases to
exhaustively list here, so we've listed only a few below. Pay attention to the corresponding decompositions. Note that these are not the only tame ones in 4 dimensions.
\begin{longtable}{m{4cm}m{3cm}m{1.25cm}m{7cm}m{1cm}}
  \hline
  Algebra & Involution & $\ast$-ideal & Corresponding decomposition & Tame?\\
  \hline
  \hline
  $(a, b) + \delta (a', b')$ \footnote{where $\delta (a', b') = (b', a') \delta$ and $\delta^2 = 0$} & Any which sends $(x, y) \mapsto (y, x)$ &
  $\langle \delta \rangle$ & Singular Value Decomposition & Y\\
  \hline
  $(a + b i) + \delta (a' + b' i)$ \footnote{\label{where-takagi-alg}where $\delta i = -i \delta$ and $\delta^2 = 0$} & $\delta \mapsto \delta, i \mapsto -i$
  & $\langle \delta \rangle$ & Autone-Takagi decomposition {\cite{horn2012matrix}} & Y\\
  \hline
  $(a + b i) + \delta (a' + b' i)$ \textsuperscript{\ref{where-takagi-alg}} & $\delta \mapsto -\delta, i \mapsto -i$ & $\langle \delta \rangle$ & Skew-symmetric Takagi  decomposition& Y\\
  \hline
  $M_2 (\mathbb{R})$ & Matrix adjugate & $\langle 1 \rangle$ & ``Symplectic spectral theorem'': Analogue of the spectral theorem for $2n \times 2n$ symplectic-self adjoint matrices under symplectic similarity & Y\footnote{Tame for similar reasons to case $(\mathbb C, \tmop{id}, \langle 1 \rangle)$}\\
  \hline
  $\mathbb R[X, Y]/(X^2, Y^2)$ & $X \mapsto -X, Y \mapsto Y$ & $\langle X \rangle$ & Spectral decomposition of $H + \varepsilon H' + o (\varepsilon)$ where $H$ and $H'$ are skew-symmetric & Y\\
  \hline
\end{longtable}

By taking tensor products of $\ast$-algebras and their ideals, one can combine features of
different decompositions. For instance:

\begin{itemize}
  \item The somewhat intricate -- but still tame! -- conjugate-less analogue of SVD for complex matrices considered in {\cite{horn}} is obtained by tensoring the representative of the SVD -- which is 4-dimensional -- with the representative for the complex-symmetric spectral theorem -- which is 2-dimensional. It should be intuitively plausible that this combines their features in a compelling way.
  \item Referring to the same paper {\cite{horn}}, the main decomposition there can be seen as the analogue of Singular Value Decomposition for matrices over the $*$-algebra of double numbers\footnote{Recall that this is the algebra $\mathbb R \oplus \mathbb R$ equipped with the involution $(a,b)\mapsto (b,a)$}. Since a matrix over the double numbers is essentially a pair of real matrices, this creates an analogue of SVD for pairs of real matrices. Equivalently, this may be obtained by using tensor products: A representative of the main decomposition of that paper can be obtained by tensoring the triple representing the Singular Value Decomposition with the one respresenting the Jordan Normal Form. The resulting triple gives an involution and an ideal of an $8$-dimensional algebra.
\end{itemize}

\section{Obstacles to a general proof}

While there isn't much literature on $\ast$-algebras, there is plenty on
$C^{\ast}$-algebras. Unfortunately, much of the machinery does not generalise
well to $\ast$-algebras, as we will now show {\cite{schmudgen2020invitation}}.
And already in 2 dimensions, most $\ast$-algebras are not $C^{\ast}$-algebras.
This situation is particularly interesting because much of the general theory
in the literature concerning $\ast$-algebras that aren't $C^{\ast}$-algebras
appears to concern algebras consisting of unbounded (that is, discontinuous)
linear operators on infinite-dimensional Hilbert spaces -- it's difficult to
apply their methods to problems like this one.

One piece of machinery in the literature is ``star representations'' or
\tmtextit{$\ast$-representations}. This is the same thing as a representation
of an algebra, but with the additional requirement that the involution $\ast$
be sent to the conjugate-transpose on complex matrices. This is already
impossible for most 2-dimensional $\ast$-algebras. We therefore have a curious
situation that while every algebra has a representation, almost no
$\ast$-algebra has a $\ast$-representation.

Much of the representation theory of topological groups is developed for
compact groups {\cite{artin1998}}. However, the unitary elements of our
$\ast$-algebras usually don't form compact groups.

We've noticed that results in the literature which \tmtextit{claim} to solve
(effectively) special cases of our conjecture often have erroneous proofs. For
instance, we've tried to find generalisations of Specht's theorem
{\cite{kaplansky2003linear,horn2012matrix}} to prove additional cases
of our conjecture, but the proofs we found in the literature for those
generalisations were wrong. We confirmed this with the authors. We will not
cite examples for obvious reasons.

Trying to generalise the conjecture may be fraught: The conjecture over arbitrary fields is false, as follows from basic Witt
theory and the failure of Sylvester's Law of Inertia (undermining a prediction of the generalised
conjecture) over such fields. In spite of these obstacles, some steps to
produce something like a theory of matrix decompositions for finite fields
have been done {\cite{guralnick}}. There is prior work unifying the
classical groups {\cite{hahn2013classical}} instead of the decompositions
(with an eye towards algebraic K-theory), but the structures considered (form
rings, form modules) are quite different, and we are not sure how to apply
those general tools here.

\section{Proofs of certain cases}

We can prove some cases of the conjecture.

\begin{definition}
  Self-adjoint matrices over a $\ast$-algebra $\mathcal{A}$ satisfy
  \tmtextbf{the cancellation property under unitary similarity} if whenever
  $A$ is unitarily similar to $A'$ and $A \oplus B$ is unitarily similar to
  $A' \oplus B'$, then $B$ is unitarily similar to $B'$. Note that the matrices $A,A',B,B'$ are understood to be self-adjoint.
\end{definition}

\begin{lemma}
  \label{additive-implies-multiplicative}If for a given local $\ast$-algebra
  $\mathcal{A}$, if the self-adjoint matrices over $\mathcal{A}$ satisfy the
  cancellation property under unitary similarity, then for all $n \in
  \mathbb{N}$ the conjecture holds for $M_n (\mathcal{A})$.
\end{lemma}

\begin{proof}
  Consider $H = \sum_{i = 1}^k P_i H P_i = \sum_{i = 1}^k Q_i H Q_i$ for
  largest $k$.
  
  Use $\tmop{im} (P_i)$ for each $i$. Clearly, $\mathcal{A}^n \cong
  \bigoplus_i \tmop{im} (P_i)$. By Kaplansky's theorem, we have that each
  $\tmop{im} (P_i)$ is a free submodule. Now take any basis for each
  $\tmop{im} (P_i)$, put the column vectors together, and then use the polar
  decomposition {\cite{kaplansky2003linear}} trick to arrive at the
  multiplicative decomposition $H = [U_1 \mid U_2 \mid \ldots \mid U_k] (E_1
  \oplus E_2 \oplus \ldots \oplus E_k) [U_1 \mid U_2 \mid \ldots \mid
  U_k]^{\ast}$.
  
  The same can be done for the $Q_i$s to arrive at $H = [V_1 \mid V_2 \mid
  \ldots \mid V_k] (F_1 \oplus \ldots \oplus F_k) [V_1 \mid V_2 \mid \ldots
  \mid V_k]^{\ast}$.
  
  Imagine for the sake of contradiction that $E_1$ is not unitarily similar to
  any $F_i$. This then produces the decomposition $\mathcal{A}^n \cong
  \bigoplus_i (\tmop{im} (Q_i) \cap \tmop{im} (P_1)) \oplus \bigoplus_i
  (\tmop{im} (Q_i) \cap \tmop{im} (P_1)^{\perp})$, which has more than $k$
  non-zero factors. This contradicts the maximality of $k$. So there must be
  some $F_i$ unitarily similar to $E_1$. Assume without loss of generality
  that this is $F_1$. Use the cancellation property to cancel $E_1$ and $F_1$
  and arrive at the fact that $\bigoplus_{i \geq 2} E_i$ is unitarily similar
  to $\bigoplus_{i \geq 2} F_i$. In a similar way, we conclude that each $E_j$
  is unitarily similar to some $F_i$. The conclusion follows.
\end{proof}

\begin{proposition}
  The above conjecture is true when the underlying algebra of $\mathcal{A}$ is
  the ring of $n \times n$ matrices $\mathcal{M}_n (\mathcal{D})$over any
  division $\ast$-algebra $\mathcal{D}$.
\end{proposition}

\begin{proof}
  Either $\mathcal{D}$ is:
  \begin{itemizedot}
    \item The real numbers with the trivial involution.
    
    \item The complex numbers with either the involution
    $\tmop{id}_{\mathbb{C}}$ or $(-)^{\ast}$.
    
    \item The quaternions with either the involution $t + x \tmmathbf{i} + y
    \tmmathbf{j} + z \tmmathbf{k} \mapsto t - x \tmmathbf{i} - y \tmmathbf{j}
    - z \tmmathbf{k}$ or $t + x \tmmathbf{i} + y \tmmathbf{j} + z \tmmathbf{k}
    \mapsto t - x \tmmathbf{i} + y \tmmathbf{j} + z \tmmathbf{k}$. Note that
    while it is true that the quaternions have infinitely many involutions, it
    admits only two up to isomorphism of $\ast$-algebras.
    {\cite{rodman2014topics}}
  \end{itemizedot}
  We verify the conjecture for each case in turn.
  
  Let $\mathcal{D}$ be the real numbers. The theorem is equivalent to the
  spectral theorem here.
  
  Let $\mathcal{D}$ be the complex numbers with the involution $(-)^{\ast}$.
  The theorem is equivalent to the spectral theorem here.
  
  Let $\mathcal{D}$ be the complex numbers with the involution
  $\tmop{id}_{\mathbb{C}}$. Note that every square complex matrix is similar
  to a complex symmetric matrix. Thus, given a complex-symmetric matrix $S$,
  take its Jordan Normal Form, and replace each Jordan block with the complex
  symmetric matrix which it's similar to -- giving $S \sim J_1 \oplus \ldots
  \oplus J_k$ where each $J_i$ is complex-symmetric. We have that $S$ is
  similar to another complex-symmetric matrix. But then $S$ is furthermore
  orthogonally similar to $J_1 \oplus \ldots \oplus J_k$, by the polar
  decomposition trick {\cite{kaplansky2003linear}}.
  
  Let $\mathcal{D}$ be the quaternions with the standard involution $t + x
  \tmmathbf{i} + y \tmmathbf{j} + z \tmmathbf{k} \mapsto t - x \tmmathbf{i} -
  y \tmmathbf{j} - z \tmmathbf{k}$. The theorem is equivalent to the spectral
  theorem here.
  
  Let $\mathcal{D}$ be the quaternions with the non-standard involution $t + x
  \tmmathbf{i} + y \tmmathbf{j} + z \tmmathbf{k} \mapsto t - x \tmmathbf{i} +
  y \tmmathbf{j} + z \tmmathbf{k}$. The proof here is the same as in the case
  of $\mathbb{C}$ equipped with the involution $\tmop{id}_{\mathbb{C}}$. We
  only need to observe that:
  \begin{itemizedot}
    \item an analogue of the Jordan Normal Form exists,
    {\cite{rodman2014topics}}
    
    \item every square quaternion matrix is similar to a complex-symmetric
    matrix,
    
    \item the polar decomposition generalises to this setting. This is
    presently proved in a MathOverflow post {\cite{mo}}.
  \end{itemizedot}
\end{proof}

\begin{lemma}
  Non-singular matrices over the quaternions equipped with the non-standard
  involution\tmtextbf{} $t + x \tmmathbf{i} + y \tmmathbf{j} + z \tmmathbf{k}
  \mapsto t - x \tmmathbf{i} + y \tmmathbf{j} + z \tmmathbf{k}$ admit polar
  decompositions.
\end{lemma}

\begin{proof}
  This would be a corollary of the statement that for every non-singular
  matrix $M$ that is Hermitian with respect to the above involution, such a
  matrix admits a polynomial $p$ with real coefficients such that $p (M)^2 =
  M$.
  
  Let $M$ be a matrix Hermitian with respect to this involution. Consider a
  representation of $M$ as an $\mathbb{R}$-matrix $\chi (M)$. We can use the
  standard technique for generalising analytic functions to matrices by way of
  Hermite interpolation. We would like though for the coefficients of the
  interpolating polynomial $p \in \mathbb{C} [z]$ (for which $p (M)^2 = M$) to
  be real numbers, otherwise we might encounter problems with
  non-commutativity. We can ensure this whenever $\chi (M)$ has no negative
  real eigenvalues, by ensuring that for every congruence
  \[ p (z) \equiv \sqrt{z}  \pmod{(z - \lambda)^n} \]
  we have another congruence of the form
  \[ p (z) \equiv \overline{\sqrt{z}}  \pmod{\left( z - \overline{\lambda}
     \right)^n} . \]
  This system of congruences becomes contradictory when $\chi (M)$ has a
  negative real eigenvalue.
  
  In the event that $M$ has a negative real eigenvalue, we may perturb $M$ in
  such a way as to eliminate its real eigenvalues. We may construct a sequence
  of approximations $M_n$ to $M$, consider the sequence $\sqrt{M_n}$, refine
  this to a convergent subsequence if need be (by way of Bolzano-Weierstrass)
  and then take the limit to obtain a square root of $M$.
\end{proof}

A corollary of the Wedderburn-Malcev theorem is that every \tmtextit{local
finite-dimensional $\mathbb{R}$-algebra} admits a vector space decomposition
$A \oplus B$ where $A \in \{ \mathbb{R}, \mathbb{C}, \mathbb{H} \}$, and $B$
consists only of nilpotent elements.

\begin{definition}
  Call an involution $(-)^{\ast}$ of a \tmtextup{local finite-dimensional
  $\mathbb{R}$-algebra} $\mathcal{A}$ \tmtextit{\tmtextbf{standard}} if over
  the subalgebra $A$ (which is a division algebra) of $\mathcal{A} = A \oplus
  B$, we have that $z^2 = - 1$ implies $z^{\ast} = - z$.
\end{definition}

Call any other involution \tmtextit{non-standard}. Notice that the only
standard involution for $\mathbb{C}$ is $a + b i \mapsto a - b i$, while the
only non-standard involution is $a + b i \mapsto a + b i$. Both involutions
for the algebra of dual numbers are standard.

\begin{lemma}
  Every unit element $x$ over a local finite-dimensional $\ast$-algebra
  $\mathbb{\mathcal{A}}$ over $\mathbb{R}$ that carries a standard involution
  admits a square root and a polar decomposition.
\end{lemma}

\begin{proof}
  We only need to verify the existence of a square root. Going from the
  existence of a square root to a polar decomposition is a fairly standard
  argument.
  
  Consider an $\mathbb{R}$-matrix $M$ representing $x^{\ast} x$. We perform
  Hermite interpolation to obtain a polynomial $p$ with only real coefficients
  for which $p (M)^2 = M$. To ensure the coefficients of $p$ are real, the
  interpolation problem should be set up so that the interpolation points are
  complex conjugates of each other. This is only possible to do if $M$ has no
  negative eigenvalues. But $M$ can't have a negative eigenvalue because
  $x^{\ast} x \tmop{mod} J (\mathcal{A})$ (where $J (\mathcal{A})$ is the
  Jacobson radical of $\mathcal{A}$) is a positive definite element of $M_n
  (\mathcal{A} / J (\mathcal{A}))$.
  
  We see now that $p (x^{\ast} x)^2 = x$.
\end{proof}

\begin{proposition}
  The above conjecture is true when the underlying algebra of $\mathcal{A}$ is
  $n \times n$ matrices over a local ring $\mathcal{B}$, and the involution is
  standard.
\end{proposition}

\begin{proof}
  By the Krull-Schmidt theorem, there exists a decomposition:
  \[ H = \sum_{i = 1}^k P_i H P_i, \]
  with
  \begin{itemizedot}
    \item $P_i \in \mathcal{M} (\mathcal{A}),$
    
    \item $P_i P_j = \delta_{i j} P_i,$
    
    \item $P_i \neq 0$.
  \end{itemizedot}
  but not necessarily with $P_i^{\ast} = P_i$. We seek to make this last
  identity true as well.
  
  Choose some $i \in \{ 1, 2, \ldots, k \}$. Since $\mathcal{A}$ is a local
  ring, we may choose a basis $\{ v_1, v_{2,} \ldots \}$ where each $v_j$
  belongs to the image of $P_i$. We now show how to construct an improved
  basis $\{ v'_1, v'_2, \ldots \}$ of $\tmop{im} (P_i)$ that is orthonormal.
  
  We choose $v'_1$ to equal $v_1 \sqrt{v_1^{\ast} v_1}^{- 1}$. This definition
  should make sense as long as $v_1^{\ast} v_1$ is a unit, because we know
  that every unit has a square root. Assume for the sake of contradiction that
  it isn't a unit. Then quotienting by the Jacobson radical, and using the
  property of \tmtextit{standard} involutions, gives that each component of
  $v_1$ is a non-unit. But then the module spanned by $v_1$ is a projective
  module which isn't free, which contradicts Kaplansky's theorem. So we
  conclude that $v_1^{\ast} v_1$ is indeed a unit and the definition of $v'_1$
  makes sense.
  
  We now must choose a value of $v'_2$ to take the place of $v_2$. To this
  end, consider the linear map $f : \tmop{im} (P_i) \rightarrow \tmop{im}
  (P_i), x \mapsto v'_1 {v'_1}^{\ast} x$. Observe that $f (f (x)) = f (x)$ for
  all $x$. From this, observe that for every vector $x$, we have that $x = f
  (x) - (x - f (x))$, where $x - f (x)$ is in the kernel of $f$. Assume that
  we have an $x \in \tmop{im} (f) \cap \ker (f)$. Then we have $x = f (y)$ for
  some $y$, and then $x = f (y) = f (f (y)) = f (x) = 0$, so $x = 0$.
  Therefore we have the decomposition $\tmop{im} (P_i) = \ker (f) \oplus
  \tmop{im} (f)$. Since $\mathcal{A}$ is local, $\ker (f)$ is a free module,
  whose every element is orthogonal to $v'_1$. Continuing in the obvious way
  produces an orthonormal basis for $\tmop{im} (P_i)$.
  
  Putting these orthonormal bases for $\tmop{im} (P_i)$ for each $i$ together
  gives that $H$ is similar to a self-adjoint matrix.
  
  We now use the polar decomposition trick to make $H$ \tmtextit{unitarily
  similar} to a self-adjoint matrix.
\end{proof}

\section{Suggestion for programme}

We propose a programme, which we think might be useful in applications:
\begin{itemizedot}
  \item Prove the conjecture above.
  
  \item Classify the complexity of the corresponding decompositions or
  canonical forms. It's common to use the three-way label \tmtextit{domestic,
  tame} and \tmtextit{wild}.
  {\cite{tamewild,belitskii2003complexity}}
  
  \item Investigate the use of polymorphism in programming languages to write
  the same numerical algorithm for multiple decompositions. This has some
  resemblance to the well-known possibility of using polymorphism in dynamic
  programming algorithms {\cite{shaikhha2019finally}}. We might limit the
  numerical algorithms to all those decompositions of low enough complexity.
  Are many numerical algorithms simply the QR algorithm
  {\cite{golub2013matrix}} in disguise, written in a polymorphic way?
\end{itemizedot}

\end{document}